\documentclass[a4paper, leqno,12pt,draft]{article}

\usepackage{amsmath}
\usepackage{amssymb}
\usepackage{amsfonts}
\usepackage{amsthm}
\usepackage{enumerate}
\usepackage{comment}
\usepackage{cite}
\usepackage[dvips]{graphicx,color}
\usepackage{delarray}
\usepackage{ascmac}
\usepackage{fancybox}

\theoremstyle{plain}
\newtheorem{theorem}{Theorem}[section]
\newtheorem{corollary}[theorem]{Corollary}
\newtheorem{definition}[theorem]{Definition}
\newtheorem{proposition}[theorem]{Proposition}
\newtheorem{lemma}[theorem]{Lemma}
\newtheorem{remark}[theorem]{Remark}
\newtheorem{example}{Example}

\numberwithin{equation}{section}
\begin{document}

\title{{\bf Hyperbolic solutions to\\Bernoulli's free boundary problem}}
\author{Antoine Henrot \and Michiaki Onodera}
\date{}
\maketitle

\begin{abstract}
\noindent
Bernoulli's free boundary problem is an overdetermined problem in which one seeks an annular domain such that the capacitary potential satisfies an extra boundary condition. 
There exist two different types of solutions called elliptic and hyperbolic solutions. 
Elliptic solutions are ``stable'' solutions and tractable by variational methods and maximum principles,  while hyperbolic solutions are ``unstable'' solutions of which the qualitative behavior is less known. 
We introduce a new implicit function theorem based on the parabolic maximal regularity, which is applicable to problems with loss of derivatives. 
Clarifying the spectral structure of the corresponding linearized operator by harmonic analysis, we prove the existence of foliated hyperbolic solutions as well as elliptic solutions in the same regularity class. 
\end{abstract}

%%%%%%%%%%%%%%%%%%%%%%%%%%%%%%%%%%%%%%%%%%%%%%%%%%%%%%%%%%%%%%%%%%%%%%%%%%%%%%%%%%%%%%%%%%%%%%%%%%%%%%%%
\section{Introduction}
\label{section-introduction}

Let $\Omega$ be a bounded domain in $\mathbb{R}^n$ and $Q>0$ a given constant. 
Bernoulli's free boundary problem asks to find an open set $A\subset\Omega$ for which the following overdetermined problem is solvable: 
\begin{equation}
\label{bernoulli}
\left\{
\begin{aligned}
 -\Delta u&=0 && \textrm{in} \quad \Omega\setminus \overline{A},\\
 u&=0 && \textrm{on} \quad \partial\Omega,\\
 u&=1 && \textrm{on} \quad \partial A,\\
 \frac{\partial u}{\partial \nu}&=Q && \textrm{on} \quad \partial A,
\end{aligned}
\right.
\end{equation}
where $\nu$ is the unit outer normal vector with respect to the annular domain $\Omega\setminus\overline{A}$. 

A physical interpretation of $u$ is the stream potential of a stationary irrotational velocity field in the plane of an incompressible inviscid fluid which circulates around a bubble $A$ of air in a given container $\Omega$. 
The extra boundary condition $\partial_\nu u=|\nabla u|=Q$ is then derived from the Bernoulli's law. 
Equation \eqref{bernoulli} also arises in a shape optimization problem in which one wants to design the optimal shape of the insulation layer of an electronic cable such that the current leakage is minimized subject to a given amount of insulation material, where $u$ stands for the electrostatic potential and $\Omega$ is the cross-section of the cable with the insulation layer $\Omega\setminus\overline{A}$. 
In potential theory, $u$ is called the capacitary potential of $A$ in $\Omega$ if the first three equations in \eqref{bernoulli} are satisfied; and thus Bernoulli's free boundary problem is regarded as an inverse problem for the capacitary potential with equi-magnitude of force field on the free boundary $\partial A$. 
For other physical backgrounds, see Flucher and Rumpf \cite{flucher_rumpf-97} and references therein. 

\bigskip

The structure of solutions to Bernoulli's free boundary problem is illustrated by the simplest situation where $\Omega$ is the unit ball $\mathbb{B}=\mathbb{B}_1$. 
Here we denote by $\mathbb{B}_r$ the ball of radius $r>0$ with center at the origin. 

\begin{example}
\label{example-radial}
{\rm 
For $\Omega=\mathbb{B}$, it is known that solutions $A$ must be concentric balls (see Alessandrini \cite{alessandrini-92} and Reichel \cite{reichel-95}). 
The capacitary potential $u_r$ of $\mathbb{B}_r$ ($0<r<1$) and its normal derivative at $|x|=r$ are
\begin{equation*}
 u_r(x)=\left\{
 \begin{aligned}
  &\frac{\log |x|}{\log r} && (n=2),\\
  &\frac{|x|^{2-n}-1}{r^{2-n}-1} && (n\geq 3),
 \end{aligned}
 \right. \qquad \frac{\partial u_r}{\partial \nu}(r)=\left\{
 \begin{aligned}
  &-\frac{1}{r\log r} && (n=2),\\
  &\frac{n-2}{r(1-r^{n-2})} && (n\geq 3).
 \end{aligned}
 \right.
\end{equation*}
Thus, for given $Q>0$, we shall find $0<r<1$ such that $Q=Q(r):=\partial_\nu u_r(r)$, where $Q(r)$ is convex in $0<r<1$ and takes its minimum at
\begin{equation*}
 r=r_\ast:=\left\{
 \begin{aligned}
  &e^{-1} && (n=2),\\
  &(n-1)^{-1/(n-2)} && (n\geq 3). 
 \end{aligned}
 \right. 
\end{equation*}
Therefore, at the critical value $Q=Q_\ast:=Q(r_\ast)$, the problem \eqref{bernoulli} has a unique solution $A=\mathbb{B}_{r_\ast}$; while for $Q>Q_\ast$ there are two solutions $\mathbb{B}_{r_1(Q)}$ and $\mathbb{B}_{r_2(Q)}$ with $r_1(Q)<r_\ast<r_2(Q)$; and no solution for $Q<Q_\ast$. 
Moreover, as $Q\to\infty$, $\mathbb{B}_{r_1(Q)}$ shrinks to the single point $x=0$, while $\mathbb{B}_{r_2(Q)}$ approaches to $\Omega=\mathbb{B}$, and we have the foliation structure
\begin{equation*}
 \{0\}\cup\partial\mathbb{B}_{r_\ast}\cup\Bigg(\bigcup_{Q>Q_\ast}\partial\mathbb{B}_{r_1(Q)}\Bigg)\cup\Bigg(\bigcup_{Q>Q_\ast}\partial\mathbb{B}_{r_2(Q)}\Bigg)=\mathbb{B}. 
\end{equation*}
The shrinking solutions $\mathbb{B}_{r_1(Q)}$ are called hyperbolic and the expanding solutions $\mathbb{B}_{r_2(Q)}$ are called elliptic, as we shall define in Definition \ref{definition-elliptic}. 
}
\end{example}

One of the interesting questions is whether such a foliation structure of solutions appears for a general convex domain $\Omega$. 
Acker \cite{acker-89_1} proved that this is true for elliptic solutions, namely there is a family of expanding solutions $\{A(Q)\}_{Q>Q_0}$ for some $Q_0>0$ such that $\bigcup_{Q>Q_0}\partial A(Q)=\Omega\setminus K$ for some compact set $K\subset\Omega$. 
In fact, Cardaliaguet and Tahraoui \cite{cardaliaguet_tahraoui-02_2} proved that $K$ can be chosen as the closure of a solution $A(Q_\ast)$ with threshold value $Q_\ast>0$. 
The proof is based on the existence of ordered elliptic solutions $A(Q)$ due to the super and subsolution method and the continuous dependence of $A(Q)$ on $Q$, where the latter is deduced from Borell's inequality in convex geometry. 
However, these arguments fail for hyperbolic solutions, and it has been open if there exists a family of hyperbolic solutions shrinking to a single point in $\Omega$. 
In fact, it was conjectured by Flucher and Rumpf \cite{flucher_rumpf-97} that such a family exists with a non-degenerate harmonic center of $\Omega$ as the limit point. 

\bigskip

In general, the existence of a solution $A$ for given $\Omega$ and $Q$ can be proved by various methods including the super and subsolution method of Beurling \cite{beurling-57} and its generalizations by Caffarelli and Spruck \cite{caffarelli_spruck-82} and Henrot and Shahgholian \cite{henrot_shahgholian-00_1,henrot_shahgholian-00_2}, a variational method by Alt and Caffarelli \cite{alt_caffarelli-81}, with penalty term by Aguilera, Alt and Caffarelli \cite{aguilera_alt_caffarelli-86}, a constrained variational method by Flucher and Rumpf \cite{flucher_rumpf-97}, and the inverse function theorem of Nash and Moser by Hamilton \cite{hamilton-82}. 

Most of the results are concerned with a class of ``well-ordered'' solutions called elliptic solutions, where a solution $A$ to Bernoulli's free boundary problem \eqref{bernoulli} is called elliptic if, roughly speaking, the infinitesimal increase of the value of $Q>0$ makes the corresponding solution $A$ to expand (see Definition \ref{definition-elliptic}). 
Indeed, the super and subsolution method only allows one to construct elliptic solutions, since the method yields a solution $A$ as the union of all subsolutions, where $A_{\rm sub}\subset\Omega$ is called a subsolution if the capacitary potential $u$ of $A_{\rm sub}$ satisfies $\partial_\nu u\leq Q$ on $\partial A_{\rm sub}$; and hence for $\tilde{Q}>Q$ the corresponding solution $\tilde{A}$ must be larger than $A$. 
Thus this method cannot produce another type of solutions called hyperbolic solutions, for which the increase of $Q>0$ makes $A$ to shrink. 

Variational solutions $A=\{u<1\}$ with $u\in H_0^1(\Omega)$ constructed in \cite{alt_caffarelli-81} as minimizers of the functional
\begin{equation*}
 J(u):=\int_\Omega |\nabla u|^2\,dx+\int_{\{u<1\}}Q^2\,dx
\end{equation*}
are also elliptic, and the argument for their regularity estimate essentially relies on the local minimality of the solutions. 
Thus the method cannot directly apply to hyperbolic solutions which appear as saddle points. 

Flucher and Rumpf \cite{flucher_rumpf-97} adopted a constrained variational method in which one minimizes the capacity ${\rm Cap}\,_\Omega(A):=\int_{\Omega\setminus A}|\nabla u|^2\,dx$ among all domains $A$ of equal volume $|A|=C$, where $u$ is the capacitary potential of $A$ in $\Omega$. 
The Euler-Lagrange equation is \eqref{bernoulli} with a Lagrange multiplier $Q>0$. 
This method would produce even hyperbolic solutions. 
However, neither the continuous dependence of solution $A$ on $Q>0$, its non-degeneracy nor hyperbolicity is derivable from its construction. 

The inverse (or implicit) function theorem is also, in principle, able to handle hyperbolic solutions, but \eqref{bernoulli} has a regularity issue called ``loss of derivatives'' and this requires the Nash-Moser method as in \cite{hamilton-82}, for which one needs quantitative estimates between several function spaces which are only, to the best of our knowledge, available for elliptic solutions. 
Furthermore, the solutions $\{A(Q)\}_{|Q-Q_0|<\varepsilon}$ constructed by this method, in general, have lower regularity than the initial state $A(Q_0)$. 

\bigskip

The novelty of this paper is the introduction of an implicit function theorem based on the parabolic maximal regularity, which enables us to handle the loss of derivatives and, in particular, produces a locally foliated family of hyperbolic solutions $\{A(Q)\}_{Q_0\leq Q<Q+\varepsilon}$ in the same regularity class as the initial state $A(Q_0)$, provided that $A(Q_0)$ is initially known to be hyperbolic. 
The corresponding result for elliptic solutions is also true in the backward direction $\{A(Q)\}_{Q_0-\varepsilon<Q\leq Q_0}$. 
In fact, the result holds for variable $Q=Q(x,t)$; and thus one may assume $\Omega=\mathbb{B}$ if $\Omega$ is a planar simply-connected domain, by a conformal mapping $f:\mathbb{B}\to\Omega$. 

To state our main result precisely, we define the little H\"older space $h^{k+\alpha}(\overline{U})$ for a domain $U\subset\mathbb{R}^n$ as the closure of the Schwartz space $\mathcal{S}(\mathbb{R}^n)$ (restricted to $\overline{U}$) of rapidly decreasing functions in the H\"older space $C^{k+\alpha}(\overline{U})$. 
In the same manner, $h^{k+\alpha}(\Gamma)$ for a closed hypersurface $\Gamma$ is defined through its local coordinates. 
%Note that the Schauder regularity estimate also holds true for $h^{k+\alpha}$ by an approximation argument. 
For a reference domain $A_0\subset\overline{A_0}\subset\Omega$ with smooth boundary $\partial A_0$ and $\rho\in h^{k+\alpha}(\partial A_0)$ with small norm, we define $A_\rho$ as the bounded domain enclosed by
\begin{equation*}
 \partial A_\rho:=\left\{\zeta+\rho(\zeta)\nu(\zeta) \mid \zeta\in\partial A_0\right\}, 
\end{equation*}
where $\nu(\zeta)$ is the unit outer normal vector to $\partial A_0$ with respect to $\Omega\setminus\overline{A_0}$. 
We will show the existence of smoothly varying solutions in the sense that
\begin{equation}
\label{regularity_condition}
 \rho\in C([0,T),h^{3+\alpha}(\partial A_0))\cap C^1([0,T),h^{2+\alpha}(\partial A_0)). 
\end{equation}
A solution $A_\rho$ to \eqref{bernoulli} for $Q=Q(x)$ is said to be non-degenerate if there is no other solution nearby $A_{\rho}$ for the same value $Q(x)$, i.e., the linearized equation has only the trivial solution (see Definition \ref{definition-nondegeneracy}). 
Moreover, if the linearized equation with positive boundary data, which corresponds to the increase of $Q$, yields a positive (negative) solution, then $A_{\rho}$ is called elliptic (hyperbolic) and monotone. 

\begin{theorem}
\label{theorem-local_foliation}
Let $\partial\Omega\in h^{2+\alpha}$ and $Q(x,t)\in h^{3+\alpha}(\overline{\Omega}\times[0,\infty))$ with $Q>0$, and let $A_{\rho_0}$ with small $\rho_0\in h^{3+\alpha}(\partial A_0)$ be a non-degenerate solution to \eqref{bernoulli} for $Q(\cdot,0)$. 
\begin{enumerate}[(A)]
\item If $A_{\rho_0}$ is hyperbolic, monotone and $\partial_tQ(x,t)>0$, then there is $T>0$ such that, for any $0\leq t<T$, \eqref{bernoulli} possesses a non-degenerate, hyperbolic and monotone solution $A(t)=A_{\rho(t)}$ for $Q(\cdot,t)$ satisfying \eqref{regularity_condition} and $\rho(0)=\rho_0$. 
\item If $A_{\rho_0}$ is elliptic, monotone and $\partial_tQ(x,t)<0$, then there is $T>0$ such that, for any $0\leq t<T$, \eqref{bernoulli} possesses a non-degenerate, elliptic and monotone solution $A(t)=A_{\rho(t)}$ for $Q(\cdot,t)$ satisfying \eqref{regularity_condition} and $\rho(0)=\rho_0$. 
\end{enumerate}
\end{theorem}

\begin{remark}
{\rm 
In both cases, $A(t)$ shrinks as $t$ increases. 
This one-sided solvability reflects that \eqref{bernoulli} has a parabolic structure, and our method reveals this fact as a spectral property of the linearized operator of an evolution equation. 
}
\end{remark}

\begin{remark}
{\rm
The initial hyperbolic solution $A_{\rho_0}$ in a slightly perturbed ball $\Omega$ can be constructed as a perturbation of the radial hyperbolic solution $\mathbb{B}_{r_1(Q)}$ in $\mathbb{B}$ by a homotopy argument: the parameter $t$ now describes the deformation of $\mathbb{B}$ to $\Omega$ and the same technique developed in the present paper is applicable. 
In fact, for $n=2$, any simply-connected bounded domain $\Omega$ is conformally mapped onto the unit disk and the deformation of $\Omega$ reduces to that of $Q$, and thus our method directly applies for the construction of $A_{\rho_0}$ in $\Omega$. 
}
\end{remark}

This paper is organized as follows. 
In Section \ref{section-ift} we introduce our implicit function theorem in a functional analytic framework, in which a parametrized family of solutions $x=x(s)$ to a functional equation $F(x,s)=0$ is characterized as a solution to a certain evolution equation. 
In Section \ref{section-linearized_problem} we formulate \eqref{bernoulli} as $F(\rho,t)=0$ and derive the linearized equation, by which we define the non-degeneracy, ellipticity and hyperbolicity of solutions $A$. 
Section \ref{section-flow_characterization} deals with the corresponding evolution equation for \eqref{bernoulli}, which is described as a non-local geometric evolution of hypersurface $\partial A(t)$ for varying $Q(\cdot,t)$. 
Furthermore, another interesting characterization is given in terms of infinitely many conserved quantities (see Theorem \ref{theorem-characterization}). 
Section \ref{section-solvability} concerns the solvability of the evolution equation by the spectral analysis of the linearized operator; and thus locally foliated hyperbolic (elliptic) solutions are constructed as stated in Theorem \ref{theorem-local_foliation}. 

%%%%%%%%%%%%%%%%%%%%%%%%%%%%%%%%%%%%%%%%%%%%%%%%%%%%%%%%%%%%%%%%%%%%%%%%%%%%%%%%%%%%%%%%%%%%%%%%%%%%%%%%
\section{Implicit function theorem}
\label{section-ift}

Our approach presented in this paper is based on the parabolic method, that is, we derive and analyze an evolution equation describing the behavior of solutions $A=A(t)$ for varying data $Q=Q(\cdot,t)$. 
This approach can be formulated as an implicit function theorem applicable to nonlinear problems with structural deficit called loss of derivatives, and thus has a common feature with the implicit function theorem of Nash \cite{nash-56} and Moser \cite{moser-61,moser-66_1,moser-66_2}. 
But the parabolic approach has the advantage that the loss of derivatives can be handled with the established theory of semigroups of linear operators, and hence intricate estimates required for the Nash-Moser method are no longer needed, and more importantly the solutions constructed by our method remain in the same regularity class. 

Let $X\subset Y\subset Z$ be a triplet of Banach spaces with continuous embeddings and consider the abstract equation
\begin{equation}
\label{implicit_form}
 F(x,s)=0 \quad (x\in X, \ s\in\mathbb{R}), 
\end{equation}
where $F$ is a $C^1$-mapping from $X\times\mathbb{R}$ to $Z$ with $F(0,0)=0$. 
If the Fr\'echet derivative $\partial_xF(0,0)\in\mathcal{L}(X,Z)$ (the space of bounded operators from $X$ to $Z$) is invertible, then for each given small data $s\in\mathbb{R}$ we can find a unique solution $x(s)\in X$ in a neighborhood of $x=0$ by the classical implicit function theorem. 
Indeed, the sequence of $X$-valued curves
\begin{equation}
\label{successive}
 x_1(s):=0, \quad x_{j+1}(s):=x_{j}(s)-\partial_xF(0,0)^{-1}F(x_j(s),s) \quad (-\varepsilon\leq s\leq\varepsilon)
\end{equation}
converges to a $C^1$-curve $x(s)$ satisfying $x(0)=0$ and $F(x(s),s)=0$. 
However, the argument would fail if we only have the regularity gain $\partial_xF(0,0)^{-1}\in\mathcal{L}(Z,Y)$, and hence $x_{j+1}(s)$ is merely $Y$-valued even if $x_j(s)$ is $X$-valued. 
This ``loss of derivatives'' happens when $\partial_xF(0,0)\in\mathcal{L}(X,Z)$ is not bijective; but it has a continuous extension to $Y$ denoted again by $\partial_xF(0,0)\in\mathcal{L}(Y,Z)$ with bounded inverse $\partial_xF(0,0)^{-1}\in\mathcal{L}(Z,Y)$. 
As we shall see in Section \ref{section-linearized_problem}, the Bernoulli problem \eqref{bernoulli} has this structure, and one would try to use the Nash-Moser scheme to overcome this regularity issue, namely, 
\begin{equation*}
 x_1(s):=0, \quad x_{j+1}(s):=x_{j}(s)-S\partial_xF(x_j(s),s)^{-1}F(x_j(s),s) \quad (-\varepsilon\leq s\leq\varepsilon), 
\end{equation*}
where $S:Y\mapsto X$ is a smoothing operator, and the inverse is taken at $(x_j(s),s)$ at each different step to accelerate the convergence speed in order to compensate for the deficit coming from the artificial operator $S$. 
Here at the compensation, one needs delicate estimates between several function spaces. 

We, instead, consider the evolution equation
\begin{equation}
\label{parabolic_form}
 x'(s)+\partial_xF(x(s),s)^{-1}\partial_sF(x(s),s)=0, \quad x(0)=0
\end{equation}
under the assumption that $\partial_xF(x,s)$ has its invertible extension in $\mathcal{L}(Y,Z)$ with
\begin{equation}
\label{inverse_of_extension}
 \partial_xF(\cdot,\cdot)\in C(U,\mathcal{L}(Y,Z)), \quad \partial_xF(x,s)^{-1}\in\mathcal{L}(Z,Y) \ \ ((x,s)\in U), 
\end{equation}
where $U\subset X\times \mathbb{R}$ is a neighborhood of $(0,0)$. 
A natural regularity condition on solutions to \eqref{parabolic_form} is
\begin{equation}
\label{natural_regularity}
 x(\cdot)\in C([0,\varepsilon),X)\cap C^1([0,\varepsilon),Y), 
\end{equation}
since the both sides of \eqref{parabolic_form} are balanced in the sense of regularity. 
Under this mild regularity condition on $x$, we can characterize solutions to \eqref{implicit_form} by \eqref{parabolic_form}. 

\begin{proposition}
\label{proposition-general_characterization}
Let $F\in C^1(X\times\mathbb{R},Z)$ satisfy $F(0,0)=0$ and assume \eqref{inverse_of_extension}, \eqref{natural_regularity} and $(x(s),s)\in U$. 
Then, $F(x(s),s)=0$ for $0\leq s<\varepsilon$ and $x(0)=0$ if and only if $x$ solves \eqref{parabolic_form}. 
\end{proposition}

\begin{proof}
Under \eqref{inverse_of_extension} and \eqref{natural_regularity}, one can verify $F(x(\cdot),\cdot)\in C^1([0,\varepsilon),Z)$ and
\begin{equation*}
 \frac{d}{ds}\left(F\left(x(s),s\right)\right)=\partial_xF(x(s),s)[x'(s)]+\partial_sF(x(s),s). 
\end{equation*}
It is then straightforward to check the equivalency of the two conditions. 
\end{proof}

The original problem \eqref{implicit_form} now reduces to the solvability of \eqref{parabolic_form} in the regularity class \eqref{natural_regularity}. 
A sufficient condition for the latter is given in terms of a spectral property of the linearized operator, due to the maximal regularity theory of Da Prato and Grisvard \cite{daprato_grisvard-79}. 
The following result can be proved by slightly modifying the proof of Angenent \cite[Theorem 2.7]{angenent-90} to treat non-autonomous equations with sufficiently small $\varepsilon>0$. 
Here ${\rm Hol}(X,Y)$ denotes the set of all $E\in\mathcal{L}(X,Y)$ such that $-E$, as an unbounded operator in $Y$, generates a strongly continuous analytic semigroup on $Y$. 

\begin{proposition}
\label{proposition-solvability}
In addition to the assumptions in Proposition \ref{proposition-general_characterization}, suppose that $X,Y$ are continuous interpolation spaces and
\begin{equation*}
 \Phi(x,s):=\partial_xF(x,s)^{-1}\partial_sF(x,s)\in C^1(U,Y), \quad \partial_x\Phi(x,s)\in{\rm Hol}(X,Y). 
\end{equation*}
Then, \eqref{parabolic_form} is uniquely solvable in \eqref{natural_regularity} for some $\varepsilon>0$. 
\end{proposition}

We carry out this program for Bernoulli's problem \eqref{bernoulli} through Sections \ref{section-linearized_problem}--\ref{section-solvability} and prove Theorem \ref{theorem-local_foliation}. 
It is noteworthy that \eqref{parabolic_form} for \eqref{bernoulli} is solvable only in one direction $s\geq0$. 
This reflects the fact that the spectrum of $\partial_x\Phi(0,0)$ is unbounded in the left half plane of $\mathbb{C}$ and thus it merely generates a semigroup but a group. 
This spectrum structure clarifies the reasoning why the classical implicit function theorem fails, or at least loses the regularity of solutions to \eqref{bernoulli}, if one considers a full neighborhood $Q_0-\varepsilon<Q<Q_0+\varepsilon$. 

%%%%%%%%%%%%%%%%%%%%%%%%%%%%%%%%%%%%%%%%%%%%%%%%%%%%%%%%%%%%%%%%%%%%%%%%%%%%%%%%%%%%%%%%%%%%%%%%%%%%%%%%
\section{Linearized problem}
\label{section-linearized_problem}

We begin with formulating Bernoulli's free boundary problem \eqref{bernoulli} as a functional equation of the form \eqref{implicit_form}. 
Here, $\Omega$ is a bounded domain with $h^{2+\alpha}$-boundary $\partial\Omega$ and $Q(x,t)\in h^{2+\alpha}(\overline{\Omega}\times[0,\infty))$. 
Let us choose a reference domain $A_0\subset\overline{A_0}\subset\Omega$ with smooth boundary $\partial A_0$, say of class $h^{4+\alpha}$, and identify $\rho\in \mathcal{U}_\gamma\subset h^{3+\alpha}(\partial A_0)$ with the perturbed domain $A_\rho$ enclosed by $h^{3+\alpha}$-boundary
\begin{equation}
\label{gammarho}
 \partial A_\rho=\left\{\zeta+\rho(\zeta)\nu_0(\zeta) \mid \zeta\in\partial A_0\right\}, 
\end{equation}
where $\nu_0=\nu_0(\zeta)$ is the unit outer normal vector to $\partial A_0$ with respect to $\Omega\setminus\overline{A_0}$ and
\begin{equation*}
 \mathcal{U}_{\gamma}:=\{\rho\in h^{3+\alpha}(\partial A_0) \mid \|\rho\|_{h^{3+\alpha}(\partial A_0)}<\gamma\}, \quad \gamma\leq a/4, 
\end{equation*}
with $0<a<{\rm dist}\,(\partial A_0,\partial \Omega)$ taken to be small such that $\theta(\zeta, r):=\zeta+r\nu_0(\zeta)$ defines a diffeomorphism from $\partial A_0\times (-a,a)$ to its image. 
Denoting by $\zeta$ and $r$ the components of the inverse map $\theta^{-1}$, i.e., $\theta^{-1}(x)=(\zeta(x),r(x))$, we see that
\begin{equation*}
 \theta_{\rho}(x):=\left\{
 \begin{array}{cl}
  \theta\left(\zeta(x),r(x)+\eta(r(x))\rho(\zeta(x))\right) & {\rm if} \ x\in\theta(\partial A_0\times(-a,a)),\\
  x & {\rm otherwise},
 \end{array}
 \right. 
\end{equation*}
defines an $h^{2+\alpha}$-diffeomorphism from $\Omega\setminus A_0$ to $\Omega\setminus A_{\rho}$, where $\eta$ is a smooth cut-off function satisfying 
\begin{equation}
\label{cutoff_function}
 \eta(r)=\left\{
 \begin{array}{cl}
  1 & (|r|\leq a/4),\\
  0 & (|r|\geq 3a/4)
 \end{array}
 \right. 
 \quad \text{and} \quad \left|\frac{d\eta}{dr}(r)\right|<\frac{4}{a}. 
\end{equation}
The diffeomorphism $\theta_\rho$ induces the pull-back and push-forward operators
\begin{equation*}
 \theta_\rho^\ast u:=u\circ\theta_\rho, \quad \theta_\ast^\rho v:=v\circ\theta_\rho^{-1}
\end{equation*}
for $u\in h^{k+\alpha}(\overline{\Omega}\setminus A_{\rho})$ and $v\in h^{k+\alpha}(\overline{\Omega}\setminus A_0)$ ($0\leq k\leq2$). 
For a given $\rho\in\mathcal{U}_\gamma$, the first three equations in \eqref{bernoulli} with $A=A_\rho$ comprise the Dirichlet problem and thus always have a unique solution $u_\rho\in h^{2+\alpha}(\overline{\Omega}\setminus A_\rho)$. 
Hence if we define
\begin{equation}
\label{f}
 F(\rho,t):=\theta_\rho^\ast\left(\frac{\partial u_\rho}{\partial \nu_\rho}-Q(x,t)\right)\in h^{1+\alpha}(\partial A_0)
\end{equation}
with the unit outer normal vector $\nu_\rho$ to $\partial A_\rho$, then $A_\rho$ is a solution to \eqref{bernoulli} for $Q=Q(x,t)$ if and only if $F(\rho,t)=0$. 

Note that so far we only used the regularity $\partial A_0\in h^{3+\alpha}$ and $\partial A_\rho\in h^{2+\alpha}$, but for the differentiation of $F$ at $\rho\neq0$, we choose $A_\rho$ as a new reference domain and thus require $\partial A_0\in h^{4+\alpha}$ and $\partial A_\rho\in h^{3+\alpha}$. 
To see the effect of a change of reference domains, we define $F_{\rho_0}(\rho,t)$ by \eqref{f} with reference domain $A_{\rho_0}$, and we extend the definitions of $\partial A_\rho$, $\theta_\rho$ and $F_{\rho_0}(\rho,t)$ to allow general vectorial perturbations $\rho: \partial A_{\rho_0}\to \mathbb{R}^n$, i.e., $\partial A_\rho=\{\zeta+\rho(\zeta)\mid \zeta\in\partial A_{\rho_0}\}$ and $\theta(\zeta, \rho)=\zeta+\rho$. 
Then, 
\begin{equation}
\label{f_slide}
 F(\rho_0+\rho,t)=\theta_{\rho_0}^\ast F_{\rho_0}\left(\frac{\theta_\ast^{\rho_0}\rho}{|\nabla N_{\rho_0}|}\nu_{\rho_0}+(\theta_\ast^{\rho_0}\rho)\tau_{\rho_0},t\right)
\end{equation}
holds for $\rho_0,\rho\in \mathcal{U}_\gamma$, where $N_{\rho_0}(x):=r(x)-\rho_0(\zeta(x))$ for $x\in\theta(\partial A_0\times(-a,a))$ and $\tau_{\rho_0}:=\theta_\ast^{\rho_0}\nu_0-|\nabla N_{\rho_0}|^{-1}\nu_{\rho_0}$ is a tangential vector field on $\partial A_{\rho_0}$. 
Indeed, $N_{\rho_0}$ defines $\partial A_{\rho_0}$ as its level set and thus
\begin{equation*}
 \nu_{\rho_0}(\theta_{\rho_0}(\zeta))\cdot\nu_0(\zeta)=\frac{\nabla N_{\rho_0}(\theta_{\rho_0}(\zeta))}{|\nabla N_{\rho_0}(\theta_{\rho_0}(\zeta))|}\cdot\frac{\nabla N_0(\zeta)}{|\nabla N_0(\zeta)|}=\frac{1}{|\nabla N_{\rho_0}(\theta_{\rho_0}(\zeta))|}. 
\end{equation*}

\begin{proposition}
\label{proposition-linearized_operator}
Let $\partial\Omega\in h^{2+\alpha}$ and $Q=Q(x,t)\in h^{2+\alpha}(\overline{\Omega}\times[0,\infty))$. 
Then, 
\begin{enumerate}[(i)]
\item $F\in C^1(\mathcal{U}_\gamma\times[0,\infty),h^{1+\alpha}(\partial A_0))$. 
\item The Fr\'echet derivative of $F$ with respect to $\rho$ is given by, for $\tilde{\rho}\in h^{3+\alpha}(\partial A_0)$, 
\begin{equation*}
 \partial_\rho F(\rho,t)[\tilde{\rho}]=\theta_\rho^\ast\left(Hp+\frac{\partial p}{\partial \nu_\rho}-\frac{\partial Q}{\partial\nu_\rho}\frac{\theta_\ast^\rho\tilde{\rho}}{|\nabla N_\rho|}+\frac{\partial}{\partial \tau_\rho}\left(\frac{\partial u_\rho}{\partial \nu_\rho}-Q(x,t)\right)\theta_\ast^{\rho}\tilde{\rho}\right), 
\end{equation*}
where $H=H_{\partial A_\rho}\in h^{1+\alpha}$ is the mean curvature of $\partial A_\rho$ normalized in such a way that $H=1-n$ if $A_\rho=\mathbb{B}$, and $p$ is a unique solution to
\begin{equation}
\label{dirichlet_problem}
 \left\{
 \begin{aligned}
  -\Delta p&=0&&\textrm{in} \quad \Omega\setminus \overline{A_\rho},\\
  p&=0&&\textrm{on} \quad \partial\Omega,\\
  p&=-\frac{\partial u_\rho}{\partial\nu_\rho}\frac{\theta_\ast^\rho\tilde{\rho}}{|\nabla N_\rho|}&&\textrm{on} \quad \partial A_\rho. 
 \end{aligned}
 \right. 
\end{equation}
Note that the capacitary potential $u_\rho$ of $A_\rho$ in $\Omega$ satisfies $\partial_{\nu_\rho} u_\rho\in h^{2+\alpha}(\partial A_\rho)$. 
\item $\partial_\rho F(\rho,t)$ has its extension in $\mathcal{L}\left(h^{2+\alpha}(\partial A_0),h^{1+\alpha}(\partial A_0)\right)$ and
\begin{equation*}
 \partial_\rho F\in C\left(\mathcal{U}_\gamma\times[0,\infty),\mathcal{L}\left(h^{2+\alpha}\left(\partial A_0\right),h^{1+\alpha}\left(\partial A_0\right)\right)\right). 
\end{equation*}
\end{enumerate}
\end{proposition}

\begin{proof}
Let us first consider the differentiability of $F$ at $\rho=0$. 
Substituting $\rho=\varepsilon\tilde{\rho}$ and the formal expansion $u_\rho=u_0+\varepsilon p+o(\varepsilon)$ into
\begin{equation*}
\left\{
\begin{aligned}
 -\Delta u_\rho&=0 && \textrm{in} \quad \Omega\setminus \overline{A_\rho},\\
 u_\rho&=0 && \textrm{on} \quad \partial\Omega,\\
 u_\rho&=1 && \textrm{on} \quad \partial A_\rho, 
\end{aligned}
\right. 
\end{equation*}
we obtain
\begin{align*}
 0&=-\Delta u_\rho(x)=-\varepsilon\Delta p(x)+o(\varepsilon) && \textrm{for} \quad x\in\Omega\setminus \overline{A_0},\\
 0&=u_\rho(x)=\varepsilon p(x)+o(\varepsilon) && \textrm{for} \quad x\in\partial\Omega,\\
 1&=u_\rho(x+\varepsilon\tilde{\rho}\nu_0)=u_0(x)+\varepsilon\frac{\partial u_0}{\partial\nu_0}(x)\tilde{\rho}+\varepsilon p(x)+o(\varepsilon)\\
 &=1+\varepsilon\frac{\partial u_0}{\partial\nu_0}(x)\tilde{\rho}+\varepsilon p(x)+o(\varepsilon) && \textrm{for} \quad x\in\partial A_0. 
\end{align*}
Thus, letting $\varepsilon\to0$, we see that $p$ satisfies \eqref{dirichlet_problem}. 
Moreover, for $x\in\partial A_0$, 
\begin{align*}
 &F(\rho,t)=\frac{\partial u_\rho}{\partial \nu_\rho}(x+\varepsilon\tilde{\rho}\nu_0)-Q(x+\varepsilon\tilde{\rho}\nu_0,t)\\
 &=\frac{\partial u_\rho}{\partial\nu_0}(x+\varepsilon\tilde{\rho}\nu_0)+\varepsilon\frac{\partial u_\rho}{\partial\tau}(x+\varepsilon\tilde{\rho}\nu_0)-Q(x+\varepsilon\tilde{\rho}\nu_0,t)+o(\varepsilon)\\
 &=\frac{\partial u_0}{\partial\nu_0}(x)+\varepsilon\frac{\partial^2 u_0}{\partial \nu_0^2}(x)\tilde{\rho}+\varepsilon\frac{\partial p}{\partial\nu_0}(x)+\varepsilon\frac{\partial u_0}{\partial\tau}(x)-Q(x,t)-\varepsilon\frac{\partial Q}{\partial\nu_0}(x,t)\tilde{\rho}+o(\varepsilon)\\
 &=F(0,t)-\varepsilon H\frac{\partial u_0}{\partial\nu_0}(x)\tilde{\rho}+\varepsilon\frac{\partial p}{\partial\nu_0}(x)-\varepsilon\frac{\partial Q}{\partial\nu_0}(x,t)\tilde{\rho}+o(\varepsilon), 
\end{align*}
where we have used the fact that $\nu_\rho$ and $\partial^2 u_0/\partial \nu_0^2$ can be represented by a tangent vector $\tau$ to $\partial A_0$ and the Laplace-Beltrami operator $\Delta_{\partial A_0}$ on $\partial A_0$ as
\begin{align*}
 &\nu_\rho=\nu_0+\varepsilon\tau+o(\varepsilon),\\
 &\begin{aligned}
  0=\Delta u_0&=\Delta_{\partial A_0}u_0+\frac{\partial^2 u_0}{\partial\nu_0^2}+H\frac{\partial u_0}{\partial\nu_0}=\frac{\partial^2 u_0}{\partial\nu_0^2}+H\frac{\partial u_0}{\partial\nu_0}. 
 \end{aligned}
\end{align*}
This shows (ii) for $\rho=0$ by letting $\varepsilon\to0$. 
The above argument can be justified by estimating the corresponding norms of the error terms. 
In view of \eqref{f_slide}, the same procedure with $A_\rho$ chosen to be the reference domain yields the differentiability of $F$ and the formula in (ii) for any $\rho$; and thus (i) follows. 
Finally, the formula in (ii) still makes sense for $\tilde{\rho}\in h^{2+\alpha}(\partial A_0)$ and this gives the extension in (iii). 
\end{proof}

Suppose that $F(\rho,t)=0$, i.e., $A_\rho$ is a solution to \eqref{bernoulli} for $Q=Q(\cdot,t)>0$. 
The extended operator $\partial_\rho F(\rho,t)$ in Proposition \ref{proposition-linearized_operator} (iii) has the bounded inverse $\partial_\rho F(\rho,t)^{-1}\in\mathcal{L}(h^{1+\alpha}(\partial A_0),h^{2+\alpha}(\partial A_0))$ if the boundary value problem
\begin{equation}
\label{linear_bernoulli}
 \left\{
 \begin{aligned}
  -\Delta p&=0&&\textrm{in} \quad \Omega\setminus \overline{A},\\
  p&=0&&\textrm{on} \quad \partial\Omega,\\
  \frac{\partial p}{\partial\nu}+\left(H+\frac{\partial_\nu Q}{Q}\right)p&=\varphi&&\textrm{on} \quad \partial A
 \end{aligned}
 \right.
\end{equation}
with $A=A_\rho$ is uniquely solvable for $\varphi\in h^{1+\alpha}(\partial A_\rho)$, since the tangential derivative in the formula of $\partial_\rho F(\rho,t)$ vanishes. 
Moreover, 
\begin{equation}
\label{inverse_operator}
 \theta_\ast^\rho\partial_\rho F(\rho,t)^{-1}[\theta_{\rho}^\ast\varphi]=-\frac{p}{Q}|\nabla N_\rho|\in h^{2+\alpha}(\partial A_{\rho}). 
\end{equation}
We shall now introduce some notions for solutions $A$ to \eqref{bernoulli} in terms of the linearized equation \eqref{linear_bernoulli}. 

\begin{definition}[Non-degeneracy]
\label{definition-nondegeneracy}
A solution $A$ to \eqref{bernoulli} is called non-degenerate if the linearized problem \eqref{linear_bernoulli} with $\varphi=0$ has only the trivial solution $p=0$. 
\end{definition}

\begin{remark}
{\rm
The non-degeneracy of $A$, in fact, guarantees the unique solvability of \eqref{linear_bernoulli} for any $\varphi$ by the Fredholm theory (see Lemma \ref{lemma-nondegeneracy}). 
}
\end{remark}

Furthermore, a classification of solutions $A$ in terms of the behavior of solutions $p$ to \eqref{linear_bernoulli} was introduced by Flucher and Rumpf \cite{flucher_rumpf-97} as an extension of Beurling's original definition in \cite{beurling-57}. 

\begin{definition}[Elliptic, hyperbolic and parabolic solutions]
\label{definition-elliptic}
A solution $A$ to \eqref{bernoulli} is called elliptic (hyperbolic) if \eqref{linear_bernoulli} has a solution $p$ for $\varphi\equiv1$ and all solutions $p$ satisfy
\begin{equation*}
 \int_{\partial A}p\,d\sigma>0 \quad (<0). 
\end{equation*}
Otherwise, $A$ is called parabolic. 
Moreover, an elliptic (hyperbolic) solution $A$ is said to be monotone if $p>0$ ($<0$) holds everywhere on $\partial A$. 
\end{definition}

\begin{remark}
{\rm 
Elliptic (hyperbolic) solutions are interpreted as volume-increasing (decreasing) solutions for an infinitesimal increase of $Q$. 
Indeed, let $A=A_0$, for simplicity, be an elliptic (hyperbolic) solution to \eqref{bernoulli} for $Q=Q_0(x)>0$, and suppose $F(\rho(t),t)=0$ for $Q(x,t)=Q_0(x)+t$. 
Since one needs to have $\partial_\rho F(0,0)[\rho'(0)]=-\partial_tF(0,0)=1$, from \eqref{inverse_operator} it follows that
\begin{equation*}
 \frac{d}{dt}\left[\int_{A_{\rho(t)}}\,dx\right]\Bigg|_{t=0}=-\int_{\partial A_0}\rho'(0)\,d\sigma=\int_{\partial A_0}\frac{p}{Q_0}\,d\sigma>0 \quad (<0). 
\end{equation*}
The monotonicity implies that $A_{\rho(t)}$ increases (decreases) in the sense of set inclusion, which is nothing but Beurling's original definition. 
}
\end{remark}

In Example in Section \ref{section-introduction}, $\mathbb{B}_{r_\ast}$ is parabolic and degenerate with one dimensional kernel and $\mathbb{B}_{r_1(Q)}$ ($\mathbb{B}_{r_2(Q)}$) is non-degenerate, hyperbolic (elliptic) and monotone (see \cite[Proposition 6]{flucher_rumpf-97}). 
In particular, the non-degeneracy, ellipticity and monotonicity of $\mathbb{B}_{r_2(Q)}$ easily follows from the following useful criterion which is a slight generalization of Acker \cite{acker-89_1} and Flucher and Rumpf \cite{flucher_rumpf-97} to variable $Q(x)$. 

\begin{proposition}
If a solution $A$ (with $C^2$-boundary $\partial A$) to \eqref{bernoulli} satisfies $H+Q+Q^{-1}\partial_\nu Q>0$ on $\partial A$, then $A$ is non-degenerate, elliptic and monotone. 
\end{proposition}

\begin{proof}
The proof is based on the maximum principle applied to $w:=p-mu$ with $m:=\min_{\partial A}p$. 
Since $u,p$ respectively satisfy \eqref{bernoulli}, \eqref{linear_bernoulli}, $w$ satisfies
\begin{equation*}
 \left\{
 \begin{aligned}
  -\Delta w&=0&&\text{in} \quad \Omega\setminus \overline{A},\\
  w&=0&&\text{on} \quad \partial\Omega,\\
  w&\geq 0&&\text{on} \quad \partial A,\\
  \frac{\partial w}{\partial\nu}+\left(H+\frac{\partial_\nu Q}{Q}\right)w&=\varphi-m\left(H+Q+\frac{\partial_\nu Q}{Q}\right)&&\text{on} \quad \partial A. 
 \end{aligned}
 \right.
\end{equation*}
Hence, $\varphi\geq0$ implies $w\equiv0$ or $m>0$; otherwise the last boundary condition is violated at minima $x_0\in\partial A$ of $w$. 
In particular, $m>0$ if $\varphi\equiv 1$. 
Moreover, for $\varphi\equiv0$, both alternatives imply $m\geq0$; and similarly $M:=\max_{\partial A}p\leq 0$ by applying the same argument to $w:=p-Mu$, and hence $p\equiv 0$. 
\end{proof}

No such a simple criterion is known for hyperbolic solutions. 
Only a perturbation result is available. 
Hereafter, Definitions \ref{definition-nondegeneracy} and \ref{definition-elliptic} are extended to any domain $A$. 
The proof is postponed to Section \ref{section-solvability} (see Lemma \ref{lemma-nondegeneracy}). 

\begin{proposition}
If $A_{\rho_0}$ is non-degenerate, hyperbolic (elliptic) and monotone, then so is $A_{\rho}$ for $\rho\in h^{3+\alpha}(\partial A_0)$ provided $\|\rho-\rho_0\|_{h^{3+\alpha}(\partial A_0)}$ is sufficiently small. 
\end{proposition}

%%%%%%%%%%%%%%%%%%%%%%%%%%%%%%%%%%%%%%%%%%%%%%%%%%%%%%%%%%%%%%%%%%%%%%%%%%%%%%%%%%%%%%%%%%%%%%%%%%%%%%%%
\section{Characterization by an evolution equation}
\label{section-flow_characterization}

If $F(\rho_0,t_0)=0$ and $A_{\rho_0}$ is non-degenerate, one would proceed to the successive approximation procedure as \eqref{successive} in order to construct a solution $\rho$ to $F(\rho,t)=0$ for $t\neq t_0$; but this would generally fail because of the loss of derivatives $\partial_\rho F(\rho_0,t_0)^{-1}F(\rho,t)\in h^{2+\alpha}(\partial A_0)$ for $\rho\in h^{3+\alpha}(\partial A_0)$ as presented in Section \ref{section-ift}. 
We can overcome this regularity issue by the parabolic approach, and by Proposition \ref{proposition-general_characterization} it reduces to the solvability of
\begin{equation}
\label{evolution_eq}
 \rho'(t)+\partial_\rho F\left(\rho(t),t\right)^{-1}\partial_t F\left(\rho(t),t\right)=0
\end{equation}
with $\rho(0)=\rho_0$ under the regularity condition \eqref{regularity_condition}. 
%In fact, \eqref{regularity_condition} is suitable not only for treating the loss of derivatives in \eqref{evolution_eq}, but also for applying the standard theory of evolution equations. 
Proposition \ref{proposition-linearized_operator} and \eqref{inverse_operator} show that, for $A(t)=A_{\rho(t)}$, \eqref{evolution_eq} is represented by geometric evolution equation
\begin{equation}
\label{flow}
\begin{aligned}
 &\displaystyle V=-\frac{p}{Q} \quad \textrm{on} \quad \partial A(t),\\
 &\hspace{5mm}\textrm{with} \quad \left\{
 \begin{aligned}
 -\Delta p&=0&&\textrm{in} \quad \Omega\setminus \overline{A(t)},\\
 p&=0&&\textrm{on} \quad \partial\Omega,\\
 \frac{\partial p}{\partial\nu}+\left(H+\frac{1}{Q}\frac{\partial Q}{\partial\nu}\right)p&=\frac{\partial Q}{\partial t}&&\textrm{on} \quad \partial A(t), 
 \end{aligned}
 \right. 
\end{aligned}
\end{equation}
where $V$ is the speed of moving surface $\partial A(t)$ in the outer normal direction $\nu$ with respect to $\Omega\setminus\overline{A(t)}$ and is represented by
\begin{equation*}
 V=\frac{\theta_\ast^{\rho(t)}\rho'(t)}{|\nabla N_{\rho(t)}|}, 
\end{equation*}
and $H=H_{\partial A(t)}$ is the normalized mean curvature of $\partial A(t)$. 
Note that the evolution speed $V$ is non-locally determined by $p$ via the boundary value problem. 

Equation \eqref{flow}, by design, characterizes a family of solutions $A(t)$ to \eqref{bernoulli} for varying $Q=Q(x,t)$. 
Furthermore, \eqref{flow} also has an interesting characterization in terms of infinitely many conserved quantities
\begin{equation*}
 m_h(t):=\int_{\partial A(t)}Q(x,t)h(x)\,d\sigma
\end{equation*}
for each $h\in H_{\partial\Omega}(\Omega\setminus A(t))$, where, for nested domains $A\subset\Omega$, $H_{\partial\Omega}(\Omega\setminus A)$ denotes the space of all harmonic functions $h$ in a neighborhood of $\Omega\setminus A$, continuous up to $\partial\Omega$ and satisfying $h=0$ on $\partial\Omega$. 
As in the previous section, $\Omega$ denotes a bounded domain and $A_0\subset\overline{A_0}\subset\Omega$ is a smooth reference domain. 

\begin{theorem}
\label{theorem-characterization}
Let $\partial\Omega\in h^{2+\alpha}$ and $Q=Q(x,t)\in h^{2+\alpha}(\overline{\Omega}\times[0,\infty))$ with $Q>0$, and let $A(0)=A_{\rho(0)}$ with $\rho(0)\in h^{3+\alpha}(\partial A_0)$ be a solution to \eqref{bernoulli} for $Q(\cdot,0)$. 
If \eqref{regularity_condition} holds and $A(t)=A_{\rho(t)}$ are non-degenerate, then the following are equivalent: 
\begin{enumerate}[(i)]
\item Each $A(t)$ is a solution to \eqref{bernoulli} for $Q(\cdot,t)$; 
\item $\{A(t)\}_{0\leq t<T}$ is a solution to \eqref{flow}; 
\item $m_h'(t)=0$ for all $h\in H_{\partial\Omega}(\Omega\setminus A(t))$. 
\end{enumerate}
\end{theorem}

\begin{proof}
Equivalency of (i) and (ii) is nothing but Proposition \ref{proposition-general_characterization}. 
For (ii) $\Rightarrow$ (iii), suppose that $\{A(t)\}$ solves \eqref{flow}. 
Then, for $h\in H_{\partial\Omega}(\Omega\setminus A(t))$, 
\begin{align*}
 m_h'(t)&=\int_{\partial A(t)}\frac{\partial Q}{\partial t}h\,d\sigma+\int_{\partial A(t)}\left(\frac{\partial Q}{\partial\nu}h+Q\frac{\partial h}{\partial\nu}\right)V\,d\sigma+\int_{\partial A(t)}QhHV\,d\sigma\\
 &=\int_{\partial A(t)}\left\{\frac{\partial Q}{\partial t}h-\left(\frac{1}{Q}\frac{\partial Q}{\partial\nu}+H\right)ph-\frac{\partial h}{\partial\nu}p\right\}\,d\sigma\\
 &=\int_{\partial A(t)}\left\{\frac{\partial Q}{\partial t}-\left(\frac{1}{Q}\frac{\partial Q}{\partial\nu}+H\right)p-\frac{\partial p}{\partial\nu}\right\}h\,d\sigma=0. 
\end{align*}
On the other hand, for (iii) $\Rightarrow$ (ii), assume that $m_h'(t)=0$, i.e., 
\begin{equation*}
 \int_{\partial A(t)}\frac{\partial Q}{\partial t}h\,d\sigma+\int_{\partial A(t)}\left(\frac{\partial Q}{\partial\nu}h+Q\frac{\partial h}{\partial\nu}+QhH\right)V\,d\sigma=0
\end{equation*}
holds for $h\in H_{\partial\Omega}(\Omega\setminus A(t))$, where $V$ is the speed of moving boundary $\partial A(t)$ in the outer normal direction with respect to $\Omega\setminus\overline{A(t)}$. 
The unique solution $p$ to the boundary value problem in \eqref{flow} satisfies
\begin{align*}
 \int_{\partial A(t)}\frac{\partial Q}{\partial t}h\,d\sigma&=\int_{\partial A(t)}\left\{\frac{\partial p}{\partial\nu}+\left(H+\frac{1}{Q}\frac{\partial Q}{\partial\nu}\right)p\right\}h\,d\sigma\\
 &=\int_{\partial A(t)}\left(Q\frac{\partial h}{\partial\nu}+QhH+\frac{\partial Q}{\partial\nu}h\right)\frac{p}{Q}\,d\sigma. 
\end{align*}
Combining these two equalities, we get
\begin{equation*}
 \int_{\partial A(t)}\left(\frac{\partial Q}{\partial\nu}h+Q\frac{\partial h}{\partial\nu}+QhH\right)\left(V+\frac{p}{Q}\right)\,d\sigma=0. 
\end{equation*}
Thus, the desired conclusion $V=-p/Q$ on $\partial A(t)$ will be obtained if there exists $h\in H_{\partial\Omega}(\Omega\setminus A(t))$ satisfying
\begin{equation*}
 \frac{\partial h}{\partial\nu}+\left(H+\frac{1}{Q}\frac{\partial Q}{\partial\nu}\right)h=V+\frac{p}{Q} \quad \textrm{on}\quad \partial A(t), 
\end{equation*}
or at least if there exists a sequence of $h_k\in H_{\partial\Omega}(\Omega\setminus A(t))$ satisfying
\begin{equation}
\label{convergence}
 \sup_{x\in\partial A(t)}\left|\left\{\frac{\partial h_k}{\partial\nu}+\left(H+\frac{1}{Q}\frac{\partial Q}{\partial\nu}\right)h_k\right\}-\left(V+\frac{p}{Q}\right)\right|\rightarrow 0 \quad (k\to\infty). 
\end{equation}
Note that, the existence of such an $h$ is almost equivalent to the non-degeneracy of $A(t)$, but additionally $h$ needs to be harmonic in a neighborhood of $\partial A(t)$. 
In order to construct $h_k$, we take a sequence of bounded domains $A_k\subset\overline{A_k}\subset A(t)$ such that $\partial A_k$ approximates $\partial A(t)$ in the $h^{3+\alpha}$ sense as $k\rightarrow\infty$. 
Denoting the mean curvature of $\partial A_k$ by $H_k$ and setting $f$ as an $h^{1+\alpha}$-extension of the function $V+p/Q$ on $\partial A(t)$ to $\mathbb{R}^n$, i.e., $f|_{\partial A(t)}=V+p/Q$, we define $h_k$ as a solution to
\begin{equation*}
 \left\{
 \begin{aligned}
  -\Delta h_k&=0&&\textrm{in} \quad \Omega\setminus \overline{A_k},\\
  h_k&=0&&\textrm{on} \quad \partial\Omega,\\
  \frac{\partial h_k}{\partial\nu}+\left(H_k+\frac{1}{Q}\frac{\partial Q}{\partial\nu}\right)h_k&=f&&\textrm{on} \quad \partial A_k. 
 \end{aligned}
 \right.
\end{equation*}
Then, since $A(t)$ is non-degenerate and this non-degeneracy is preserved under a small deformation (see Lemma \ref{lemma-nondegeneracy}), we see that $A_k$ are non-degenerate and there is a uniform (in $k$) constant $C>0$ such that
\begin{equation*}
 \|h_k\|_{h^{2+\alpha}(\overline{\Omega}\setminus A_k)}\leq C\|f\|_{h^{1+\alpha}(\overline{\Omega}\setminus A_k)}\leq C\|f\|_{h^{1+\alpha}(\mathbb{R}^n)}. 
\end{equation*}
Now \eqref{convergence} follows from this uniform estimate and the mean value theorem. 
\end{proof}

In the special case $\Omega=\mathbb{B}$, a special set of conserved quantities are enough for the characterization. 
Let $\mathcal{H}_k$ be the vector space of all homogeneous harmonic polynomials of degree $k\in\mathbb{N}\cup\{0\}$ on $\mathbb{R}^n$. 
The dimension $d_k^{(n)}$ of $\mathcal{H}_k$ is given by $d_0^{(n)}=1$, $d_1^{(n)}=n$ and 
\begin{equation*}
 d_k^{(n)}=\left(
 \begin{array}{c}
  k+n-1\\
  k
 \end{array}
 \right)-\left(
 \begin{array}{c}
  k+n-3\\
  k-2
 \end{array}
 \right) \quad (k\geq 2). 
\end{equation*}
We recall that, for $h_1\in\mathcal{H}_{k_1}$ and $h_2\in\mathcal{H}_{k_2}$ with $k_1\neq k_2$, $h_1$ and $h_2$ are orthogonal to each other with respect to the $L^2(\partial\mathbb{B})$ inner product, i.e., 
\begin{equation*}
  \langle h_1, h_2 \rangle_{L^2}:=\int_{\partial\mathbb{B}}h_1h_2\,d\sigma=0. 
\end{equation*}
Moreover, by choosing an orthonormal basis $\{h_{k,1}, h_{k,2}, \ldots, h_{k,d_k^{(n)}}\}$ of $\mathcal{H}_k$, it is known that
\begin{equation*}
 \bigcup_{k=0}^\infty \{h_{k,1}, h_{k,2}, \ldots, h_{k,d_k^{(n)}}\}
\end{equation*}
forms an orthonormal basis of $L^2(\partial\mathbb{B})$. 
Now we define the weighted moments $m_{k,l}(t)$ of $\partial A(t)$ for $k\in\mathbb{N}\cup\{0\}$ and $l=1,2,\ldots,d_k^{(n)}$ by
\begin{align*}
 m_{k,l}(t)&:=\int_{\partial A(t)}Q(x,t)H_{k,l}(x)\,d\sigma,\\
 H_{k,l}(x)&:=\left(1-|x|^{2-n-2k}\right)h_{k,l}(x) \quad (k\geq 1),\\
 H_{0,1}(x)&:=\left\{
 \begin{aligned}
  & \log|x| && (n=2),\\
  & 1-|x|^{2-n} && (n\geq 3). 
 \end{aligned}
 \right.
\end{align*}
Note that $H_{k,l}$ is the difference between $h_{k,l}$ and its Kelvin transform; and thus harmonic in $\mathbb{R}^n\setminus\{0\}$ and $H_{k,l}\in H_{\partial\mathbb{B}}(\mathbb{B}\setminus\{0\})$. 

\begin{corollary}
\label{corollary-first_integrals}
Let $\Omega=\mathbb{B}$ and assume that $A_0$ is connected and $0\in\bigcup_{0\leq s<T}A(s)$. 
Under the same assumption in Theorem \ref{theorem-characterization}, the conditions (i)--(iii) in Theorem \ref{theorem-characterization} are also equivalent to
\begin{enumerate}
\item[(iv)] $m_{k,l}(t)=m_{k,l}(0)$ for all $k\in\mathbb{N}\cup\{0\}$ and $l=1,\ldots,d_k^{(n)}$. 
\end{enumerate}
\end{corollary}

\begin{proof}
It suffices to prove that $h\in H_{\partial\mathbb{B}}(\mathbb{B}\setminus A(t))$ can be approximated uniformly on $\partial A(t)$ by a linear combination of $H_{k,l}$'s. 
Since $A(t)$ is topologically equivalent to $A_0$, Runge's theorem (see Armitage and Gardiner \cite[Theorem 2.6.4]{armitage_gardiner-01} for a higher dimensional version) deduces that, for any $\varepsilon>0$, $|h-H|<\varepsilon$ holds on $\overline{\mathbb{B}}\setminus A(t)$ for some harmonic function $H$ in $\mathbb{R}^n\setminus\{0\}$. 
Adding a harmonic function $\tilde{H}$ in $\mathbb{B}$ with $\tilde{H}=-H$ on $\partial\mathbb{B}$, we see that $H+\tilde{H}\in H_{\partial\mathbb{B}}(\mathbb{B}\setminus\{0\})$ and $|H+\tilde{H}-h|<2\varepsilon$ on $\overline{\mathbb{B}}\setminus A(t)$. 
Moreover, $H+\tilde{H}$ is harmonically extended to $\mathbb{R}^n\setminus\{0\}$ by the Kelvin transform and the Schwarz reflection principle; and thus it has a unique Laurent expansion of the form
\begin{equation*}
 H+\tilde{H}=\sum_{k=0}^\infty\sum_{l=1}^{d_k^{(n)}}\alpha_{k,l}h_{k,l}+\sum_{k=0}^\infty\sum_{l=1}^{d_k^{(n)}}\beta_{k,l}|x|^{2-n-2k}h_{k,l} \quad (n\geq 3). 
\end{equation*}
Here, $H+\tilde{H}=0$ on $\partial\mathbb{B}$ implies $\alpha_{k,l}+\beta_{k,l}=0$. 
Thus, 
\begin{equation*}
 \left|\sum_{k=0}^K\sum_{l=1}^{d_k^{(n)}}\alpha_{k,l}H_{k,l}-h\right|\leq 3\varepsilon \quad \text{on} \ \partial A(t)
\end{equation*}
for sufficiently large $K$. 
The case $n=2$ is similar, with a logarithmic term. 
\end{proof}

Characterization (iv) can be thought of as a quadrature identity or moment conservation law and it leads to an interesting connection to integrable systems as the Hele-Shaw flow (see Sakai \cite{sakai-82}, Gustafsson and Shapiro \cite{gustafsson_shapiro-05}). 
In fact, as we shall state below, we can also give a characterization of solution $A$ to the ``stationary'' problem \eqref{bernoulli} itself by a quadrature identity. 
Theorem \ref{theorem-characterization} can be alternatively proved by using this identity. 

\begin{proposition}
Let $Q(x)\in h^{1+\alpha}(\overline{\Omega})$ and $\partial\Omega,\partial A\in h^{2+\alpha}$. 
Then, the following are equivalent: 
\begin{enumerate}[(a)]
\item $A$ is a solution to Bernoulli's free boundary problem \eqref{bernoulli}; 
\item For any $h\in H_{\partial\Omega}(\Omega\setminus A)$, 
\begin{equation}
\label{quadrature_identity}
 \int_{\partial A}Qh\,d\sigma=\int_{\partial A}\frac{\partial h}{\partial\nu}\,d\sigma. 
\end{equation}
\end{enumerate}
Moreover, if one of the conditions above holds, the solution $u$ to \eqref{bernoulli} is given by
\begin{equation}
\label{green_representation}
 u(x)=\int_{\partial A}Q(y)G_\Omega(x,y)\,d\sigma(y), 
\end{equation}
where $G_\Omega$ denotes the Green's function of $\Omega$ for the Dirichlet-Laplacian. 
\end{proposition}

\begin{proof}
The assertion (a) $\Rightarrow$ (b) follows from integration by parts: 
\begin{equation*}
 \int_{\partial A}Qh\,d\sigma=\int_{\partial A}\frac{\partial u}{\partial\nu}h\,d\sigma=\int_{\partial A}\frac{\partial h}{\partial\nu}\,d\sigma. 
\end{equation*}
For (b) $\Rightarrow$ (a), it suffices to check that $u$ defined by \eqref{green_representation} solves \eqref{bernoulli}. 
Since
\begin{equation*}
\left\{
\begin{aligned}
 -\Delta u&=Q\mathcal{H}^{n-1}\lfloor\partial A && \textrm{in} \quad \Omega, \\
 u&=0 && \textrm{on} \quad \Omega, 
\end{aligned}
\right.
\end{equation*}
where $\mathcal{H}^{n-1}\lfloor\partial A$ is the $(n-1)$-dimensional Hausdorff measure restricted to $\partial A$, and this singular measure causes the derivative jump
\begin{equation*}
 \lim_{z\notin\overline{A}, z\to x}\frac{\partial u}{\partial\nu}(z)-\lim_{z\in A, z\to x}\frac{\partial u}{\partial\nu}(z)=Q(x) \quad (x\in\partial A), 
\end{equation*}
it remains to show that $u=1$ on $\partial A$ (which implies $u=1$ on $\overline{A}$). 
But this can be checked by choosing $h(y)=G_\Omega(z,y)$ with $z\in A$ in \eqref{quadrature_identity} and then taking the limit $z\rightarrow x\in\partial A$. 
\end{proof}

\begin{remark}
{\rm 
All of the preceding arguments also work for $C^{k+\alpha}$ instead of $h^{k+\alpha}$. 
}
\end{remark}

%%%%%%%%%%%%%%%%%%%%%%%%%%%%%%%%%%%%%%%%%%%%%%%%%%%%%%%%%%%%%%%%%%%%%%%%%%%%%%%%%%%%%%%%%%%%%%%%%%%%%%%%
\section{Existence of foliated solutions}
\label{section-solvability}

We shall complete our program in Section \ref{section-ift} for Bernoulli's problem \eqref{bernoulli} and thus prove Theorem \ref{theorem-local_foliation}. 
By Theorem \ref{theorem-characterization}, the remaining task is to prove the local-in-time solvability of evolution equation \eqref{flow} in the regularity class \eqref{regularity_condition}. 
Since $h^{k+\alpha}(\partial A_0)$ is a continuous interpolation space, it suffices to check that
\begin{equation}
\label{def_of_phi}
 \Phi(\rho,t):=\partial_\rho F(\rho,t)^{-1}\partial_tF(\rho,t)
\end{equation}
meets the required conditions in Proposition \ref{proposition-solvability}. 
Thus Theorem \ref{theorem-local_foliation} is obtained as a consequence of the following proposition, in which
\begin{equation*}
 \mathcal{U}_\varepsilon(\rho_0):=\left\{\rho\in\mathcal{U}_\gamma \mid \|\rho-\rho_0\|_{h^{3+\alpha}(\partial A_0)}<\varepsilon\right\}. 
\end{equation*}
Note that we require the higher regularity $Q\in h^{3+\alpha}$ as compared to the previous section to ensure that $\Phi\in C^1$. 

\begin{proposition}
\label{proposition-generation}
Let $\partial\Omega\in h^{2+\alpha}$, $Q\in h^{3+\alpha}(\overline{\Omega}\times[0,\infty))$, $Q>0$ and $\rho_0\in\mathcal{U}_\gamma$. 
Then, for sufficiently small $\varepsilon>0$, $\Phi \in C^1(\mathcal{U}_\varepsilon(\rho_0)\times[0,\varepsilon),h^{2+\alpha}(\partial A_0))$ and
\begin{equation*}
 \partial_{\rho}\Phi(\rho,t)\in {\rm Hol}(h^{3+\alpha}(\partial A_0),h^{2+\alpha}(\partial A_0))
\end{equation*}
if one of the following conditions holds true: 
\begin{enumerate}[(A)]
\item $A_{\rho_0}$ is non-degenerate, hyperbolic and monotone for $Q(\cdot,0)$, and $\partial_tQ>0$; 
\item $A_{\rho_0}$ is non-degenerate, elliptic and monotone for $Q(\cdot,0)$, and $\partial_t Q<0$. 
\end{enumerate}
\end{proposition}

\begin{remark}
{\rm 
Depending on the hyperbolicity/ellipticity of $A_{\rho}$, the linearized operator has the opposite sign, which reflects in assumption $\partial_tQ\gtrless 0$. 
}
\end{remark}

In order to prove Proposition \ref{proposition-generation}, it is more convenient to use a functional analytic representation of \eqref{evolution_eq} than the geometric representation \eqref{flow}. 
Let $A_0\subset\overline{A_0}\subset\Omega$ be a smooth reference domain and define $\mathcal{U}_\gamma$, $A_\rho$, $\theta_\rho$ and $N_\rho$ as in Section \ref{section-linearized_problem}. 
Using the pull-back and push-forward operators $\theta_\rho^\ast, \theta_\ast^\rho$, we define
\begin{align*}
 L(\rho)v&:=\theta_\rho^\ast(-\Delta)\theta_\ast^\rho v\in h^{\alpha}(\overline{\Omega}\setminus A_0),\\
 B(\rho,t)v&:=\left(\left\langle\left(\theta_\rho^\ast\nabla\theta_\ast^\rho v\right)|_{\partial A_0}, \nu_\rho\right\rangle+M_vH(\rho)+M_{v(\theta_\rho^\ast Q)^{-1}}\left\langle\left(\theta_\rho^\ast\nabla Q\right)|_{\partial A_0}, \nu_\rho\right\rangle, v|_{\partial\Omega}\right)\\
 &\hspace{8cm}\in h^{1+\alpha}(\partial A_0)\times h^{2+\alpha}(\partial\Omega),\\
 S(\rho,t)f&:=(L(\rho),B(\rho,t))^{-1}(f,0,0)\in h^{2+\alpha}(\overline{\Omega}\setminus A_0),\\
 T(\rho,t)\varphi&:=(L(\rho),B(\rho,t))^{-1}(0,\varphi,0)\in h^{2+\alpha}(\overline{\Omega}\setminus A_0)
\end{align*}
for $v\in h^{2+\alpha}(\overline{\Omega}\setminus A_0)$, $f\in h^{\alpha}(\overline{\Omega}\setminus A_0)$ and $\varphi\in h^{1+\alpha}(\partial A_0)$, where $\langle\cdot,\cdot\rangle$ denotes the inner product in $\mathbb{R}^n$, $M_\psi$ is the pointwise multiplication operator defined by 
\begin{equation*}
 (M_\varphi \psi)(\zeta):=\varphi(\zeta)\psi(\zeta) \quad (\zeta\in\partial A_0), 
\end{equation*}
$H(\rho)\in h^{1+\alpha}(\partial A_0)$ assigns the mean curvature of $\partial A_\rho$ at $\theta_\rho(\zeta)$ to each $\zeta\in\partial A_0$, and $\nu_\rho\in h^{2+\alpha}(\partial A_0,\mathbb{R}^n)$ denotes the outer normal vector field represented by
\begin{equation*}
 \nu_\rho(\zeta):=\frac{\nabla N_\rho(\theta_\rho(\zeta))}{|\nabla N_\rho(\theta_\rho(\zeta))|} \quad (\zeta\in\partial A_0). 
\end{equation*}
The following lemma shows that the solution operators $S(\rho,t)$, $T(\rho,t)$ are well-defined whenever $A_\rho$ is non-degenerate, and moreover under the condition in Proposition \ref{proposition-generation}, $\Phi(\rho,t)$ is defined in $\mathcal{U}_{\varepsilon}(\rho_0)\times[0,\varepsilon)$ for small $\varepsilon>0$. 

\begin{lemma}
\label{lemma-nondegeneracy}
Let $\rho,\rho_0\in\mathcal{U}_{\gamma}$. 
Then, 
\begin{enumerate}[(i)]
\item $(L(\rho),B(\rho,t))\in\mathcal{L}(h^{2+\alpha}(\overline{\Omega}\setminus A_0),h^\alpha(\overline{\Omega}\setminus A_0)\times h^{1+\alpha}(\partial A_0)\times h^{2+\alpha}(\partial\Omega))$ is invertible if and only if $A_{\rho}$ is non-degenerate for $Q(\cdot,t)$. 
\item If $A_{\rho_0}$ is non-degenerate, hyperbolic (elliptic) and monotone for $Q(\cdot,0)$, then so is $A_{\rho}$ for $Q(\cdot,t)$, if $\rho\in \mathcal{U}_\varepsilon(\rho_0)$ and $0\leq t<\varepsilon$ for some small $\varepsilon>0$. 
\end{enumerate}
\end{lemma}

\begin{proof}
Let us first observe that $(L(\rho), B_\mu(\rho,t))$ is invertible for large $\mu>0$, where 
\begin{equation}
\label{b_mu}
 B_\mu(\rho,t)v:=B(\rho,t)v+(\mu v|_{\partial A_0}, 0). 
\end{equation}
Indeed, this corresponds to the unique solvability of 
\begin{equation}
\label{coercivity_of_bilinear_form}
 \left\{
 \begin{aligned}
  -\Delta p&=f && \textrm{in} \quad \Omega\setminus \overline{A_{\rho}},\\
  p&=\varphi_1 && \textrm{on} \quad \partial\Omega,\\
  \frac{\partial p}{\partial\nu}+\left(H+\frac1Q\frac{\partial Q}{\partial\nu}+\mu\right)p&=\varphi_2 && \textrm{on} \quad \partial A_{\rho}
 \end{aligned}
 \right. 
\end{equation}
in $h^{2+\alpha}(\overline{\Omega}\setminus A_{\rho})$, and this follows if $H+Q^{-1}\partial_\nu Q+\mu>0$ on $\partial A_{\rho}$, since the complementing boundary condition of Agmon, Douglis and Nirenberg \cite{agmon_douglis_nirenberg-64} is satisfied (see also Gilbarg and Trudinger \cite[Theorem 6.31]{gilbarg_trudinger-98}). 
Now take any $p\in h^{2+\alpha}(\overline{\Omega}\setminus A_0)$ and set $v:=(L(\rho),B_\mu(\rho,t))^{-1}(L(\rho),B(\rho,t))p$ and subtract the corresponding equations as \eqref{coercivity_of_bilinear_form} satisfied by $p$ and $v$ to find that
\begin{equation*}
 \left\{\left(L(\rho),B_\mu(\rho,t)\right)^{-1}\left(L(\rho),B(\rho,t)\right)-I\right\}p=-\mu\left(B(\rho),B_\mu(\rho,t)\right)^{-1}(0,p,0)=:Kp, 
\end{equation*}
where $I$ is the identity and $K$ is a compact operator on $h^{2+\alpha}(\overline{\Omega}\setminus A_{0})$. 
Hence (i) follows from the Fredholm alternative applied to
\begin{equation*}
 I+K=(L(\rho), B_\mu(\rho,t))^{-1}(L(\rho), B(\rho,t)): 
\end{equation*}
the operator $(L(\rho),B(\rho,t))$ is bijective if and only if it is injective, i.e., if $A_{\rho}$ is non-degenerate. 

The non-degeneracy statement in (ii) readily follows from the fact that the set of invertible operators is open in the space of bounded operators with the operator norm topology, since $(\rho,t)\mapsto (L(\rho),B(\rho,t))$ is continuous. 
The hyperbolicity (ellipticity) and monotonicity then follow from the continuity of $(\rho,t)\mapsto T(\rho,t)[1]$. 
\end{proof}

\begin{remark}
{\rm
The solvability of \eqref{coercivity_of_bilinear_form} can also be seen from the corresponding bilinear form
\begin{equation*}
 \mathfrak{B}(p,\phi):=\int_{\Omega\setminus A_{\rho}}\nabla p\cdot\nabla\phi\,dx+\int_{\partial A_{\rho}}\left(H+\frac1Q\frac{\partial Q}{\partial\nu}+\mu\right)p\phi\,d\sigma, 
\end{equation*}
where we assume $\varphi_1\equiv0$ for simplicity (but the general case reduces to this case). 
Indeed, $\mathfrak{B}(p,\phi)$ is coercive in the Sobolev space $H^1_0(\Omega)$ (restricted to $\Omega\setminus A_{\rho}$) if $H+Q^{-1}\partial_\nu Q+\mu>0$. 
}
\end{remark}

Using the above notations, we can rewrite \eqref{def_of_phi} as
\begin{equation*}
 \Phi(\rho,t)=M_{|\theta_\rho^\ast(\nabla N_{\rho})|(\theta_\rho^\ast Q)^{-1}}T(\rho,t)\theta_\rho^\ast\partial_tQ. 
\end{equation*}
Note that the regularity of $\Phi$ depends on that of $Q$, namely, 
\begin{equation*}
 \Phi\in C^l(\mathcal{U}_{\varepsilon}(\rho_0)\times[0,\varepsilon),h^{2+\alpha}(\partial A_0))
\end{equation*}
if $Q\in h^{l+2+\alpha}(\overline{\Omega}\times[0,\infty))$, where $l\in\mathbb{N}$, $l=\infty$ or $l=\omega$. 

\begin{proof}[Proof of Proposition \ref{proposition-generation}]
\underline{\it Step 1.} 
In the following proof, we write $\partial$ for $\partial_{\rho}$ and $h^{k+\alpha}$ for $h^{k+\alpha}(\partial A_0)$ for brevity. 
Let us first derive a representation of $\partial\Phi(\rho,t)$. 
For this purpose, we recall that the mean curvature operator $H(\rho)$ has a quasilinear structure (see Escher and Simonett {\cite[Lemma 3.1]{escher_simonett-97_2}}) as
\begin{equation*}
 H(\rho)=P(\rho)\rho+K(\rho)
\end{equation*}
with $P\in C^{\omega}(\mathcal{V},\mathcal{L}(h^{3+\alpha},h^{1+\alpha}))$ and $K\in C^{\omega}(\mathcal{V},h^{1+\alpha})$, where
\begin{equation*}
 \mathcal{V}:=\{\rho\in h^{2+\alpha}(\partial A_0)\mid \|\rho\|_{C^1(\partial A_0)}<a/4\}. 
\end{equation*}
Hence, 
\begin{equation}
\label{linearized_b}
 \partial B(\rho,t)[\tilde{\rho}]v=\partial\left(B(\rho,t)v\right)[\tilde{\rho}]=\left(M_{v}P(\rho)\tilde{\rho}+L_1\tilde{\rho},0\right), 
\end{equation}
where the linear operator $L_1=L_1(\rho,t,v)$ is of lower-order in the sense that $L_1\in\mathcal{L}(h^{2+\alpha},h^{1+\alpha})$. 
By differentiating the identity
\begin{equation*}
 \left(L(\rho)T(\rho,t)\varphi, B(\rho,t)T(\rho,t)\varphi\right)=(0,\varphi,0) \quad \left(\varphi\in h^{1+\alpha}\right) 
\end{equation*}
with respect to $\rho\in\mathcal{U}_{\gamma}$ and using \eqref{linearized_b}, we see that
\begin{align*}
 L(\rho)\partial T(\rho,t)[\tilde{\rho}]\varphi&=-\partial L(\rho)[\tilde{\rho}]T(\rho,t)\varphi,\\
 B(\rho,t)\partial T(\rho,t)[\tilde{\rho}]\varphi&=-\partial B(\rho,t)[\tilde{\rho}]T(\rho,t)\varphi=\left(-M_{T(\rho,t)\varphi}P(\rho)\tilde{\rho}-L_1\tilde{\rho},0\right). 
\end{align*}
Hence, using the solution operators $S(\rho,t)$ and $T(\rho,t)$, we find the representation
\begin{align*}
 \partial T(\rho,t)[\tilde{\rho}]\varphi
 &=-S(\rho,t)\partial L(\rho)[\tilde{\rho}]T(\rho,t)\varphi-T(\rho,t)M_{T(\rho,t)\varphi}P(\rho)\tilde{\rho}-T(\rho,t)L_1\tilde{\rho}\\
 &=-T(\rho,t)M_{T(\rho,t)\varphi}P(\rho)\tilde{\rho}+L_2\tilde{\rho}
\end{align*}
for $\varphi=\theta_\rho^\ast q\in h^{1+\alpha}$ with a lower-order operator $L_2=L_2(\rho,t)\in\mathcal{L}(h^{2+\alpha},h^{2+\alpha})$. 
Therefore, we deduce that
\begin{equation*}
 \partial\Phi(\rho,t)[\tilde{\rho}]=-M_1(\rho,t)T(\rho,t)M_2(\rho,t)P(\rho)\tilde{\rho}+L_3\tilde{\rho}+L_4\tilde{\rho}+L_5\tilde{\rho}+L_6\tilde{\rho}, 
\end{equation*}
where
\begin{align*}
 &M_1(\rho,t):=M_{|\theta_\rho^\ast(\nabla N_{\rho})|(\theta_\rho^\ast Q)^{-1}} \in \mathcal{L}\left(h^{2+\alpha},h^{2+\alpha}\right),\\
 &M_2(\rho,t):=M_{T(\rho,t)\theta_\rho^\ast \partial_tQ}\in\mathcal{L}\left(h^{1+\alpha},h^{1+\alpha}\right),\\
 &L_3=L_3(\rho,t):=M_1(\rho,t)L_2(\rho,t)\in\mathcal{L}\left(h^{2+\alpha},h^{2+\alpha}\right),\\
 &L_4=L_4(\rho,t):=\partial M_{|\theta_\rho^\ast(\nabla N_{\rho})|}[\cdot]M_{(\theta_\rho^\ast Q)^{-1}}T(\rho,t)\theta_\rho^\ast \partial_tQ\in\mathcal{L}\left(h^{3+\alpha}, h^{2+\alpha}\right),\\
 &L_5=L_5(\rho,t):=M_{|\theta_\rho^\ast(\nabla N_{\rho})|}\partial M_{(\theta_\rho^\ast Q)^{-1}}[\cdot]T(\rho,t)\theta_\rho^\ast \partial_tQ\in\mathcal{L}\left(h^{2+\alpha},h^{2+\alpha}\right),\\
 &L_6=L_6(\rho,t):=M_1(\rho,t)T(\rho,t)\partial(\theta_\rho^\ast \partial_tQ)[\cdot] \in \mathcal{L}\left(h^{2+\alpha},h^{2+\alpha}\right). 
\end{align*}
Here, $L_3, L_5, L_6$ are lower-order operators as compared to the principal part
\begin{equation}
\label{pi-definition}
 \Pi(\rho,t):=-M_1(\rho_1,t)T(\rho,t)M_2(\rho,t)P(\rho) \in \mathcal{L}(h^{3+\alpha},h^{2+\alpha}). 
\end{equation}
Moreover, $L_4$ can also be regarded as a negligible perturbation if $\gamma$ is sufficiently small, since a direct computation shows that
\begin{equation*}
 \|L_4\tilde{\rho}\|_{h^{2+\alpha}}\leq \delta(\gamma)\|\tilde{\rho}\|_{h^{3+\alpha}}+C\|\tilde{\rho}\|_{C^{3}} \quad \left(\rho\in\mathcal{U}_{\gamma}\right)
\end{equation*}
holds with $\delta(\gamma)\to 0$ as $\gamma\to 0$. 
The required smallness of $\gamma$ apparently depends on $\rho_0$; but we can choose a new smooth reference domain $A_0$ arbitrarily close to $A_{\rho_0}$ in the $h^{3+\alpha}$ sense and thus $\gamma$ can be independently chosen. 
Since ${\rm Hol}(h^{3+\alpha},h^{2+\alpha})$ is open in $\mathcal{L}(h^{3+\alpha},h^{2+\alpha})$ and the operator norms of $L_3,L_4,L_5,L_6$ are arbitrarily small, it suffices to prove that
\begin{equation*}
 \Pi(\rho_0,0)\in {\rm Hol}(h^{3+\alpha},h^{2+\alpha}). 
\end{equation*}
Moreover, we may replace $T(\rho_0,0)$ in \eqref{pi-definition} by
\begin{equation*}
 \tilde{T}(\rho_0,0):=(L(\rho_0),B_\mu(\rho_0,0))^{-1}(0,\cdot,0)
\end{equation*}
with a large constant $\mu>0$ such that $H+Q^{-1}\partial_\nu Q+\mu>0$ on $\partial A_{\rho_0}$ (see \eqref{b_mu} for the definition of $B_\mu$). 
Indeed, $\tilde{T}(\rho_0,0)$ is well-defined and the elliptic regularity estimate applied to $(T(\rho_0,0)-\tilde{T}(\rho_0,0))\varphi=\mu T(\rho_0,0)\tilde{T}(\rho_0,0)\varphi$ yields
\begin{equation}
\label{difference_of_t}
\begin{aligned}
 \|(T(\rho_0,0)-\tilde{T}(\rho_0,0))\varphi\|_{h^{2+\alpha}}
 &\leq C\mu\|\tilde{T}(\rho_0,0)\varphi\|_{h^{1+\alpha}}\\
 &\leq C\mu\|\tilde{T}(\rho_0,0)\varphi\|_{h^{2+\beta}}\\
 &\leq C\mu\|\varphi\|_{h^{1+\beta}}
\end{aligned}
\end{equation}
with $0<\beta<\alpha<1$; and thus the difference
\begin{align*}
 &\Pi(\rho_0,0)+M_1(\rho_0,0)\tilde{T}(\rho_0,0)M_2(\rho_0,0)P(\rho_0)\\
 &\hspace{1cm}=M_1(\rho_0,0)\left(\tilde{T}(\rho_0,0)-T(\rho_0,0)\right)M_2(\rho_0,0)P(\rho_0)\in\mathcal{L}(h^{3+\beta},h^{2+\alpha})
\end{align*}
is of lower-order. 
We are now led to prove that
\begin{equation}
\label{tildepi_is_generator}
 \tilde{\Pi}(\rho_0,0):=-M_1(\rho_0,0)\tilde{T}(\rho_0,0)M_2(\rho_0,0)P(\rho_0)\in{\rm Hol}(h^{3+\alpha},h^{2+\alpha}). 
\end{equation}
Under the assumption that (A) or (B) holds (see Definition \ref{definition-elliptic}), we have
\begin{equation*}
 T(\rho_0,0)\theta_{\rho_0}^\ast \partial_tQ<0. 
\end{equation*}
Thus $M_1(\rho_0,0)$ and $-M_2(\rho_0,0)$ are multiplication operators with uniformly positive functions on $\partial A_0$. 

\underline{\it Step 2.} 
In order to prove \eqref{tildepi_is_generator}, it is sufficient to prove the resolvent estimate
\begin{equation}
\label{resolvent_estimate}
 |\lambda|\|\tilde{\rho}\|_{h^{2+\alpha}}+\|\tilde{\rho}\|_{h^{3+\alpha}}\leq C\|(\lambda+\tilde{\Pi}(\rho_0,0))\tilde{\rho}\|_{h^{2+\alpha}}
\end{equation}
for $\tilde{\rho}\in h^{3+\alpha}$ and $\lambda\in\{z\in\mathbb{C}\mid {\rm Re}\,z\geq\lambda_\ast\}$ for some $\lambda_\ast>0$. 
We begin with the analysis of a localized version of the operator $\tilde{\Pi}(\rho_0,0)$ at each point $\zeta\in\partial A_0$, i.e., 
\begin{equation*}
 \tilde{\Pi}(\rho_0,0,\zeta):=-m_1m_2\mathcal{T}_0(-\Delta_{\mathbb{R}^{n-1}}+1)\in\mathcal{L}(h^{3+\alpha}(\mathbb{R}^{n-1}),h^{2+\alpha}(\mathbb{R}^{n-1})), 
\end{equation*}
where $m_1:=M_1(\rho_0,0)(\zeta)>0$, $m_2:=M_2(\rho_0,0)(\zeta)<0$, $\Delta_{\mathbb{R}^{n-1}}$ is the $(n-1)$-dimensional Laplace operator, and $\mathcal{T}_0\in\mathcal{L}(h^{1+\alpha}(\mathbb{R}^{n-1}),h^{2+\alpha}(\mathbb{R}^{n-1}))$ maps $\varphi\in h^{1+\alpha}(\mathbb{R}^{n-1})$ to (the trace of) a unique solution $v=\mathcal{T}_0\varphi\in h^{2+\alpha}(\mathbb{R}^{n-1})$ to
\begin{equation*}
 \left\{
 \begin{aligned}
  -\Delta v&=0&&\textrm{in} \quad \mathbb{R}^n_+:=\{(x',x_n)\in\mathbb{R}^n \mid x_n>0\},\\
  -\frac{\partial v}{\partial n}+\mu_\zeta v&=\varphi&&\textrm{on} \quad \mathbb{R}^{n-1}=\partial\mathbb{R}^n_+, 
 \end{aligned}
 \right.
\end{equation*}
with a positive constant $\mu_\zeta:=H(\rho_0)(\zeta)+\theta_{\rho_0}^\ast(Q^{-1}\partial_\nu Q)(\zeta)+\mu$. 
Applying the (partial) Fourier transformation $\mathcal{F}$ on $\mathbb{R}^{n-1}$, we have
\begin{equation*}
 \left\{
 \begin{aligned}
  |\xi|^2\mathcal{F}v-\frac{\partial^2 \mathcal{F}v}{\partial n^2}&=0&&\textrm{for} \quad (\xi,x_n)\in\mathbb{R}^n_+,\\
  -\frac{\partial\mathcal{F}v}{\partial n}+\mu_\zeta\mathcal{F}v&=\mathcal{F}\varphi&&\textrm{for} \quad \xi\in\mathbb{R}^{n-1}. 
 \end{aligned}
 \right.
\end{equation*}
For fixed $\xi\in\mathbb{R}^{n-1}$, this is a second-order ordinary differential equation and one can easily obtain the explicit formula
\begin{equation*}
 \mathcal{F}v(\xi,0)=\frac{1}{|\xi|+\mu_\zeta}\mathcal{F}\varphi(\xi). 
\end{equation*}
Therefore, 
\begin{equation}
\label{multiplier}
 \tilde{\Pi}(\rho_0,0,\zeta)=-m_1m_2\mathcal{F}^{-1}\frac{|\xi|^2+1}{|\xi|+\mu_\zeta}\mathcal{F}
\end{equation}
with $-m_1m_2>0$. 
This representation combined with a Mikhlin-type multiplier theorem for the little H\"older spaces allows one to obtain
\begin{equation}
\label{local_estimate}
 |\lambda|\|\tilde{\rho}\|_{2+\alpha}+\|\tilde{\rho}\|_{3+\alpha}\leq C\|(\lambda+\tilde{\Pi}(\rho_0,0,\zeta))\tilde{\rho}\|_{2+\alpha}
\end{equation}
for $\tilde{\rho}\in h^{3+\alpha}(\mathbb{R}^{n-1})$ and $\lambda\in\{z\in\mathbb{C}\mid {\rm Re}\,z>0\}$, where $\|\cdot\|_{k+\alpha}:=\|\cdot\|_{h^{k+\alpha}(\mathbb{R}^{n-1})}$. 

\underline{\it Step 3.} 
The last step is to derive \eqref{resolvent_estimate} from the local estimate \eqref{local_estimate} by showing that $\tilde{\Pi}(\rho_0,0)$ can be indeed approximated by $\tilde{\Pi}(\rho_0,0,\zeta)$ in a small neighborhood $U_{\zeta}$ of $\zeta\in\partial A_0$. 
At this point, $\Delta$ in the definition of $\mathcal{T}_0$ and $\Delta_{\mathbb{R}^{n-1}}$ should have been replaced respectively by constant coefficient elliptic operators
\begin{equation*}
 \mathcal{L}_0:=\sum_{j,k=1}^n a_{jk}\frac{\partial^2}{\partial x_j\partial x_k}, \quad \mathcal{P}_0:=\sum_{j,k=1}^{n-1}p_{jk}\frac{\partial^2}{\partial x_j \partial x_k}
\end{equation*}
unless $\partial A_0$ is flat near $\zeta\in\partial A_0$; but still a similar representation as \eqref{multiplier} can be obtained by an algebraic consideration (see Onodera \cite[Section 3.4]{onodera-15} for the details). 
As in the proof of \cite[Lemma 6]{onodera-15}, we can prove that, for any $\varepsilon>0$, there exist an atlas $\{(U_{\zeta_l},\psi_l)\}_{l=1}^m$ of the tubular domain $\theta(\partial A_0,(-d,0])$ with small $d>0$ and an associated partition of unity $\{\phi_l\}_{l=1}^m$ such that ${\rm supp}\,\phi_l\subset U_{\zeta_l}$, $\bigcup_{l=1}^m\phi_l=1$ on $\theta(\partial A_0,(-d/2,0])$ and
\begin{equation}
\label{commutator_estimate}
 \left\|\psi_l^\ast M_{\phi_l}\tilde{\Pi}(\rho_0,0)\tilde{\rho}-\tilde{\Pi}(\rho_0,0,\zeta_l)\psi_l^\ast M_{\phi_l}\tilde{\rho}\right\|_{2+\alpha}\leq\varepsilon\left\|\psi_l^\ast M_{\phi_l}\tilde{\rho}\right\|_{3+\alpha}+C\|\tilde{\rho}\|_{h^{3+\beta}}
\end{equation}
holds for $\tilde{\rho}\in h^{3+\alpha}$, $1\leq l\leq m$ and $0<\beta<\alpha<1$. 
This commutator estimate is essentially based on Leibniz' rule and the smallness of $d>0$, where the latter improves the accuracy of the approximations by constant coefficient operators. 
The difference between our operator $\tilde{\Pi}(\rho_0,0)$ and $-W(\rho)$ studied in \cite{onodera-15} is the presence of the Dirichlet boundary condition on $\partial\Omega$ in the definition of $\tilde{T}(\rho_0,0)$. 
But, this does not affect the derivation of \eqref{commutator_estimate}, since it is the homogeneous condition $v|_{\partial\Omega}\equiv 0$. 
Now we combine \eqref{local_estimate} and \eqref{commutator_estimate} with small $\varepsilon>0$ to obtain
\begin{equation*}
 |\lambda|\|\psi_l^\ast M_{\phi_l}\tilde{\rho}\|_{2+\alpha}+\|\psi_l^\ast M_{\phi_l}\tilde{\rho}\|_{3+\alpha}
 \leq C\left(\|\psi_l^\ast M_{\phi_l}(\lambda+\tilde{\Pi}(\rho_0,0))\tilde{\rho}\|_{2+\alpha}+\|\tilde{\rho}\|_{h^{3+\beta}}\right)
\end{equation*}
for $\tilde{\rho}\in h^{3+\alpha}$, $\lambda\in\{z\in\mathbb{C}\mid {\rm Re}\,z>0\}$ and $1\leq l\leq m$. 
Since 
\begin{equation*}
 \tilde{\rho} \mapsto \max_{1\leq l \leq m}\|\psi_l^\ast M_{\phi_l}\tilde{\rho}\|_{k+\alpha}
\end{equation*}
defines an equivalent norm on $h^{k+\alpha}$, the above inequality implies 
\begin{equation*}
 |\lambda|\|\tilde{\rho}\|_{h^{2+\alpha}}+\|\tilde{\rho}\|_{h^{3+\alpha}}
 \leq C\left(\|(\lambda+\tilde{\Pi}(\rho_0,0))\tilde{\rho}\|_{h^{2+\alpha}}+\|\tilde{\rho}\|_{h^{3+\beta}}\right). 
\end{equation*}
Therefore, \eqref{resolvent_estimate} follows from this inequality and the interpolation inequality
\begin{equation*}
 \|\tilde{\rho}\|_{h^{3+\beta}}\leq\epsilon\|\tilde{\rho}\|_{h^{3+\alpha}}+C\|\tilde{\rho}\|_{h^{2+\alpha}}
\end{equation*}
by choosing a sufficiently large $\lambda_\ast>0$. 
\end{proof}

\bigskip

\noindent
{\bf Acknowledgments.}
The second author was supported in part by the Grant-in-Aid for Young Scientists (B) 16K17628, JSPS. 

%%%%%%%%%%%%%%%%%%%%%%%%%%%%%%%%%%%%%%%%%%%%%%%%%%%%
%%%%%%%%%%%%%%%%%%%%%%%%%%%%%%%%%%%%%%%%%%%%%%%%%%%%

\end{document}